\documentclass[12pt]{article} 
\usepackage{amsfonts,amsmath,latexsym,amssymb,mathrsfs,amsthm}
\usepackage{slashbox}
\usepackage{comment}

\let\OLDthebibliography\thebibliography
\renewcommand\thebibliography[1]{
  \OLDthebibliography{#1}
  \setlength{\parskip}{0pt}
  \setlength{\itemsep}{0pt plus 0.3ex}
}

\def\numberlikeadb{\global\def\theequation{\thesection.\arabic{equation}}}
\numberlikeadb
\newtheorem{theorem}{Theorem}[section]
\newtheorem{lemma}[theorem]{Lemma}

\newtheorem{proposition}[theorem]{Proposition}
\newtheorem{remark}[theorem]{Remark}

\usepackage{color}

\usepackage{lscape}
\usepackage{caption}
\usepackage{multirow}
\begin{document}

\title{Functional inequalities and monotonicity results for modified Lommel functions of the first kind}
\author{Robert E. Gaunt\footnote{Department of Mathematics, The University of Manchester, Oxford Road, Manchester M13 9PL, UK}}

\date{\today} 
\maketitle

\vspace{-5mm}

\begin{abstract}We establish some monotonicity results and functional inequalities for modified Lommel functions of the first kind.  In particular, we obtain new Tur\'{a}n type inequalities and bounds for ratios of modified Lommel functions of the first kind, as well as the function itself.  These results complement and in some cases improve on existing results, and also generalise a number of the results from the literature on monotonicity patterns and functional inequalities for the modified Struve function of the first kind.
\end{abstract}

\noindent{{\bf{Keywords:}}} Modified Lommel function; bounds; Tur\'{a}n type inequalities; monotonicity

\noindent{{{\bf{AMS 2010 Subject Classification:}}} Primary 33C20; 26D07

\section{Introduction}\label{intro}

The modified Lommel function of the first kind $t_{\mu,\nu}(x)$ is a particular solution to the modified Lommel differential equation \cite{g29}. Modified lommel functions arise in scattering amplitudes in quantum optics \cite{t73}, stress distributions in cylindrical objects \cite{s85} and the physics of two-dimensional diffusions \cite{t56} and heat conduction \cite{b61}.  The modified Lommel function $t_{\mu,\nu}(x)$ generalises the modified Strive function of the first kind $\mathbf{L}_\nu(x)$ (see Section \ref{sec2}), which also arises in manifold applications; see \cite{bp13} for a list of application areas. 

Over the last several decades an extensive literature has built up on monotonicity results and functional inequalities for Bessel, modified Bessel and related functions motivated through problems in the applied sciences (see, for example, \cite{baricz2, segura} and references therein); however, only recently have such results started to be obtained for modified Lommel functions.  In \cite{m17} some monotonicity properties and convexity results for the modified Lommel function of the first kind $t_{\mu,\nu}(x)$ were obtained, from which some Tur\'{a}n type inequalities followed. A Redheffer type bound for the function $t_{\mu-1/2,1/2}(x)$ was also obtained by \cite{m17}, and \cite{css18} established accurate bounds for $t_{\mu,\nu}(x)$, although these bounds only hold for $0<x<1$.  In a very recent work, \cite{gaunt ineq7} extended results of \cite{gaunt ineq5} concerning functional inequalities for modified Struve functions of the first kind to obtain bounds for the important quantities $t_{\mu,\nu}(x)/t_{\mu-1,\nu-1}(x)$, $xt_{\mu,\nu}'(x)/t_{\mu,\nu}(x)$, $t_{\mu,\nu}(x)/t_{\mu,\nu}(y)$ and the function $t_{\mu,\nu}(x)$ itself in terms of analogous expressions involving the modified Bessel functions of the first kind $I_\nu(x)$.  These results are quite powerful because there is substantial literature on functional inequalities for modified Bessel functions from which one can draw suitable bounds (see \cite{baricz2, gaunt ineq5, segura} and references therein).  

Our aim in this paper is to further contribute to the recent literature on functional inequalities and monotonicity properties of modified Lommel functions of the first kind, as well as to generalise existing results for the modified Struve function of the first kind.  In Section \ref{sec3.1}, we generalise results from the comprehensive study of monotonicity properties and functional inequalities for the modified Struve function $\mathbf{L}_\nu(x)$ given by \cite{bp14}, which in turn complemented and improved results of \cite{jn98}.  As the functions $t_{\mu,\nu}(x)$ and $\mathbf{L}_\nu(x)$ share a similar power series representation (see Section \ref{sec2} for these and further properties), the approach of \cite{bp14}, which involves appealing to general results on the monotonicity of quotients of power series, is also effective for our purpose of modified Lommel functions of the first kind, and we are also able to obtain some other monotonicity results and functional inequalities, which complement results of \cite{gaunt ineq7}. We note that a few of the monotonicity and convexity results of \cite{bp14} have already been generalised by \cite{m17}.  In addition, in Section \ref{sec3.2}, we obtain new Tur\'{a}n type inequalities for the modified Lommel function of the first kind, which complement a different type of Tur\'{a}n type inequalities for $t_{\mu,\nu}(x)$ that were obtained by \cite{m17}.  One of our Tur\'{a}n type inequalities generalises one of \cite{bp14} given for the modified Struve function $\mathbf{L}_\nu(x)$, whilst our other two-sided inequality gives a new Tur\'{a}n type inequality in the special case of the modified Struve function $\mathbf{L}_\nu(x)$, and therefore complements the results of \cite{bps17} and \cite{bb16}.  We also note that our Tur\'{a}n type inequalities naturally complement those of \cite{bk16} that were given for a certain type of Lommel function of the first kind.



\section{The modified Lommel function of the first kind}\label{sec2}

The modified Lommel function of the first kind $t_{\mu,\nu}(x)$ is defined by the hypergeometric series
\begin{align}\label{hypt}t_{\mu,\nu}(x)=2^{\mu-1}\Gamma\big(\tfrac{\mu-\nu+1}{2}\big)\Gamma\big(\tfrac{\mu+\nu+1}{2}\big)\sum_{k=0}^\infty\frac{(\frac{1}{2}x)^{\mu+2k+1}}{\Gamma\big(k+\frac{\mu-\nu+3}{2}\big)\Gamma\big(k+\frac{\mu+\nu+3}{2}\big)},\nonumber
\end{align}
and arises as a particular solution of the modified Lommel differential equation \cite{s36,r64}
\begin{equation}\label{lommelode}x^2f''(x)+xf'(x)-(x^2+\nu^2)f(x)=x^{\mu+1}.
\end{equation}
In the literature different notation is used for the modified Lommel functions; we use the notation of \cite{zs13}.  The terminology modified Lommel function of the \emph{first kind} is also not standard, but has recently been introduced by \cite{gaunt ineq7}. 

We shall follow \cite{gaunt ineq7} and use the following normalization which will be useful for our purposes, as it will remove a number of multiplicative constants from our calculations:
\begin{align}\label{tseries}\tilde{t}_{\mu,\nu}(x)&=\frac{1}{2^{\mu-1}\Gamma\big(\frac{\mu-\nu+1}{2}\big)\Gamma\big(\frac{\mu+\nu+1}{2}\big)}t_{\mu,\nu}(x)\\
&=\sum_{k=0}^\infty\frac{(\frac{1}{2}x)^{\mu+2k+1}}{\Gamma\big(k+\frac{\mu-\nu+3}{2}\big)\Gamma\big(k+\frac{\mu+\nu+3}{2}\big)}.\nonumber
\end{align}
To ease the exposition, we will also refer to $\tilde{t}_{\mu,\nu}(x)$ as the modified Lommel function of the first kind.  From now on, we shall work with the function $\tilde{t}_{\mu,\nu}(x)$; results for $t_{\mu,\nu}(x)$ can be easily inferred.  We note the important special case
\begin{equation*}\tilde{t}_{\nu,\nu}(x)= \mathbf{L}_\nu(x),
\end{equation*}
where $\mathbf{L}_\nu(x)$ is a modified Struve function of the first kind.

For $x>0$, the function $ \tilde{t}_{\mu,\nu}(x)$ is positive if $\mu-\nu\geq-3$ and $\mu+\nu\geq-3$ (equivalently $\mu\geq-3$ and $|\nu|\leq\mu+3$). The function $ \tilde{t}_{\mu,\nu}(x)$ satisfies the following recurrence relations \cite{gaunt ineq7} and differentiation formula \cite{g20}:
\begin{align}\label{Iidentity}\tilde{t}_{\mu-1,\nu-1}(x)-\tilde{t}_{\mu+1,\nu+1}(x)&=\frac{2\nu}{x}\tilde{t}_{\mu,\nu}(x)+a_{\mu,\nu}(x), \\
\label{lomrel2}\tilde{t}_{\mu-1,\nu-1}(x)+\tilde{t}_{\mu+1,\nu+1}(x)&=2\tilde{t}_{\mu,\nu}'(x)-a_{\mu,\nu}(x), \\
\label{diffone1}\frac{\mathrm{d}}{\mathrm{d}x}\bigg(\frac{\tilde{t}_{\mu,\nu} (x)}{x^\nu}\bigg)&=
\frac{\tilde{t}_{\mu+1,\nu+1} (x)}{x^\nu}+\frac{a_{\mu,\nu}(x)}{x^\nu},
\end{align}
where 
\begin{equation}\label{anot}a_{\mu,\nu}(x)=\frac{(x/2)^\mu}{\Gamma(\frac{\mu-\nu+1}{2})\Gamma(\frac{\mu+\nu+3}{2})}.
\end{equation} 
It will be also useful to follow \cite{gaunt ineq7} and introduce the following function:
\begin{equation*}b_{\mu,\nu}(x):=\frac{xa_{\mu,\nu}(x)}{2 \tilde{t}_{\mu,\nu}(x)}=\frac{(\frac{1}{2}x)^{\mu+1}}{\Gamma\big(\frac{\mu-\nu+1}{2}\big)\Gamma\big(\frac{\mu+\nu+3}{2}\big) \tilde{t}_{\mu,\nu}(x)}.
\end{equation*}
The function $ \tilde{t}_{\mu,\nu}(x)$ has the following asymptotic properties  \cite{gaunt ineq7}:
\begin{align}\label{ttend0} \tilde{t}_{\mu,\nu}(x)&\sim\frac{(\frac{1}{2}x)^{\mu+1}}{\Gamma\big(\frac{\mu-\nu+3}{2}\big)\Gamma\big(\frac{\mu+\nu+3}{2}\big)}, \quad x\downarrow0,\:\mu>-3,\:|\nu|<\mu+3, \\
\label{Itendinfinity} \tilde{t}_{\mu,\nu}(x)&\sim\frac{\mathrm{e}^x}{\sqrt{2\pi x}}\bigg(1-\frac{4\nu^2-1}{8x}\bigg), \quad x\rightarrow\infty, \:\mu,\nu\in\mathbb{R}.
\end{align}

We end this section by recording that the modified Bessel function of the first kind and modified Struve function of the first kind are defined by the power series \cite{olver}
\begin{align}\label{Idef}I_\nu(x)&=\sum_{k=0}^\infty\frac{(\frac{1}{2}x)^{2k+\nu}}{k!\Gamma(k+\nu+1)}, \\
\label{Ldef} \mathbf{L}_\nu(x)&=\sum_{k=0}^\infty\frac{(\frac{1}{2}x)^{2k+\nu+1}}{\Gamma(k+\frac{3}{2})\Gamma(k+\nu+\frac{3}{2})}.
\end{align}
The functions $I_\nu(x)$ and $\mathbf{L}_\nu(x)$ have the following asymptotic behaviour:
\begin{align}\label{itend0}I_\nu(x)&\sim \frac{x^\nu}{2^\nu\Gamma(\nu+1)}, \quad x\downarrow0,\:\nu\not=-1,-2,-3\ldots,\\
\label{Ltend0}\mathbf{L}_\nu(x)&\sim \frac{x^{\nu+1}}{\sqrt{\pi}2^\nu\Gamma(\nu+\frac{3}{2})}, \quad x\downarrow0,\:\nu\not=-\tfrac{3}{2},-\tfrac{5}{2},-\tfrac{7}{2}\ldots, \\
\label{itendinf}I_\nu(x)&\sim\frac{\mathrm{e}^x}{\sqrt{2\pi x}}, \quad \mathbf{L}_\nu(x)\sim\frac{\mathrm{e}^x}{\sqrt{2\pi x}}, \quad x\rightarrow\infty,\:\nu\in\mathbb{R}.
\end{align}

\section{Main results and proofs}\label{sec3}

\subsection{Monotonicity results and associated inequalities}\label{sec3.1}

In the sequel we shall need the following result (see \cite{bk55,pv97}).  

\begin{lemma}\label{lem2.1}Suppose the power series $f(x) =\sum_{n\geq0}a_nx^n$ and $g(x) =\sum_{n\geq0}b_nx^n$, where $a_n \in\mathbb{R}$ and $b_n > 0$ for all $n\geq0$, both converge on $(-r, r)$, $r > 0$. If $\{a_n/b_n\}_{n\geq0}$ is an increasing (decreasing) sequence, then the function $x \mapsto f(x)/g(x)$ is also increasing (decreasing) on $(0, r)$.
\end{lemma}

For our purposes, it is important to note that Lemma \ref{lem2.1} also holds when both the power series $f(x)$ and $g(x)$ are even, or both are odd functions.

\begin{theorem}\label{thm1}The following assertions are true:

\begin{itemize}

\item[(i)] If $-\frac{3}{2}<\nu<\mu$ ($|\nu|<\mu+3$ for $-3<\mu\leq-\frac{3}{2}$, and $\mu<\nu<\mu+3$ for $\mu>-\frac{3}{2}$), then the function $x\mapsto x^{\mu-\nu}\mathbf{L}_\nu(x)/\tilde{t}_{\mu,\nu}(x)$ is strictly increasing (decreasing) on $(0,\infty)$.  

\item[(ii)] If $\mu>-\frac{3}{2}$ and $|\mu|>|\nu|$ ($\mu>-\frac{3}{2}$ and $|\mu|<|\nu|<\mu+3$), then the function $x\mapsto \mathbf{L}_\mu(x)/\tilde{t}_{\mu,\nu}(x)$ is strictly increasing (decreasing) on $(0,\infty)$.

\item[(iii)] If $\mu>-2$ and $|\nu|<|\mu+1|$ ($\mu>-2$ and $|\mu+1|<|\nu|<\mu+3$), then the function $x\mapsto I_{\mu+1}(x)/\tilde{t}_{\mu,\nu}(x)$ is strictly increasing (decreasing) on $(0,\infty)$.

\item[(iv)] If $\mu\geq\mu_1>-3$, $|\nu|<\mu+3$, $|\nu_1|<\mu_1+3$ and $(\mu-\mu_1)(\mu+\mu_1+6)\geq\nu^2-\nu_1^2$ ($\mu_1\geq\mu>-3$, $|\nu|<\mu+3$, $|\nu_1|<\mu_1+3$ and $(\mu_1-\mu)(\mu_1+\mu+6)\geq\nu^2-\nu_1^2$), then the function $x\mapsto x^{\mu_1-\mu}\tilde{t}_{\mu,\nu}(x)/\tilde{t}_{\mu_1,\nu_1}(x)$ is increasing (decreasing) on $(0,\infty)$.  If either of the weak inequalities is strict, then the function is strictly increasing (decreasing).

\item[(v)] If $\mu>-3$ and $|\nu|<\mu+3$, then the function $x\mapsto x\tilde{t}_{\mu,\nu}'(x)/\tilde{t}_{\mu,\nu}(x)$ is strictly increasing on $(0,\infty)$.

\item[(vi)] Fix $\mu$ and $\nu$ such that $\mu>-3$, $|\nu|<\mu+3$. Then $\delta\mapsto \tilde{t}_{\mu+\delta+1,\nu+\delta+1}/\tilde{t}_{\mu+\delta,\nu+\delta}(x)$ is a strictly decreasing function on $(0,\infty)$.

\item[(vii)] Let $x>0$.  Then
\begin{equation}\label{monineq1}\tilde{t}_{\mu,\nu}(x)<\frac{\sqrt{\pi}2^{\nu-\mu-1}\Gamma(\nu+\frac{3}{2})}{\Gamma(\frac{\mu-\nu+3}{2})\Gamma(\frac{\mu+\nu+3}{2})}x^{\mu-\nu}\mathbf{L}_\nu(x), \quad -\tfrac{3}{2}<\nu<\mu,
\end{equation}
and the inequality is reversed for $|\nu|<\mu+3$ for $-3<\mu\leq-\frac{3}{2}$, and $\mu<\nu<\mu+3$ for $\mu>-\frac{3}{2}$;
\begin{equation}\label{monineq2}\tilde{t}_{\mu,\nu}(x)<\frac{\sqrt{\pi}\Gamma(\mu+\frac{3}{2})}{2\Gamma(\frac{\mu-\nu+3}{2})\Gamma(\frac{\mu+\nu+3}{2})}\mathbf{L}_\mu(x), \quad \mu>-\tfrac{3}{2}, \:|\mu|>|\nu|,
\end{equation}
and the inequality is reversed for $\mu>-\frac{3}{2}$, $|\mu|<|\nu|<\mu+3$;
\begin{equation}\label{monineq3}\tilde{t}_{\mu,\nu}(x)<\frac{\Gamma(\mu+2)}{\Gamma(\frac{\mu-\nu+3}{2})\Gamma(\frac{\mu+\nu+3}{2})}I_{\mu+1}(x), \quad \mu>-2,\: |\nu|<|\mu+1|,
\end{equation}
and the inequality is reversed for $\mu>-2$, $|\mu+1|<|\nu|<\mu+3$; the inequality
\begin{equation}\label{monineq3.5} 2^\mu\Gamma\big(\tfrac{\mu-\nu+3}{2}\big)\Gamma\big(\tfrac{\mu+\nu+3}{2}\big)x^{-\mu}\tilde{t}_{\mu,\nu}(x)\geq 2^{\mu_1}\Gamma\big(\tfrac{\mu_1-\nu_1+3}{2}\big)\Gamma\big(\tfrac{\mu_1+\nu_1+3}{2}\big)x^{-\mu_1}\tilde{t}_{\mu_1,\nu_1}(x), 
\end{equation}
holds for $\mu\geq\mu_1>-3$, $|\nu|<\mu+3$, $|\nu_1|<\mu_1+3$, $(\mu-\mu_1)(\mu+\mu_1+6)\geq\nu^2-\nu_1^2$, and the inequality is reversed for $\mu_1\geq\mu>-3$, $|\nu|<\mu+3$, $|\nu_1|<\mu_1+3$, $(\mu_1-\mu)(\mu_1+\mu+6)\geq\nu^2-\nu_1^2$ (inequality (\ref{monineq3.5}) is strict if any of the weaker inequalities involving $\mu$ and $\nu$ are strict);
\begin{equation}\label{monineq3.75}\frac{x\tilde{t}_{\mu,\nu}'(x)}{\tilde{t}_{\mu,\nu}(x)}>\mu+1, \quad \mu>-3, \: |\nu|<\mu+3,
\end{equation}
and
\begin{equation}\label{monineq4}\bigg|\frac{x\tilde{t}_{\mu,\nu}'(x)}{\tilde{t}_{\mu,\nu}(x)}\bigg|<\sqrt{x^2+\nu^2+2(\mu+\nu+1)b_{\mu,\nu}(x)}, \quad \mu>-3, \: |\nu|<\mu+3.
\end{equation}
All these inequalities are sharp in the limit $x\downarrow0$.

\end{itemize}

\end{theorem}

\begin{proof} (i) From the power series representations (\ref{Ldef}) and (\ref{tseries}) of $\mathbf{L}_\nu(x)$ and $\tilde{t}_{\mu,\nu}(x)$, respectively, we can write
\begin{equation*}\frac{x^{\mu-\nu}\mathbf{L}_\nu(x)}{\tilde{t}_{\mu,\nu}(x)}=2^{\mu-\nu}\frac{\sum_{k=0}^\infty\alpha_{\nu,k}x^{2k}}{\sum_{k=0}^\infty\beta_{\mu,\nu,k} x^{2k}},
\end{equation*}
where
\begin{equation}\label{betaterm} \alpha_k=\frac{1}{\Gamma(k+\frac{3}{2})\Gamma(k+\nu+\frac{3}{2})}\quad \text{and} \quad \beta_{\mu,\nu,k}=\frac{1}{\Gamma\big(k+\frac{\mu-\nu+3}{2}\big)\Gamma\big(k+\frac{\mu+\nu+3}{2}\big)}.
\end{equation}
We now let $q_k=\alpha_{\nu,k}/\beta_{\mu,\nu,k}$ and use the standard formula $\Gamma(x+1)=x\Gamma(x)$ to calculate 
\begin{equation*}\frac{q_{k+1}}{q_k}-1=\frac{\big(k+\frac{\mu-\nu+3}{2}\big)\big(k+\frac{\mu+\nu+3}{2}\big)}{(k+\frac{3}{2})(k+\nu+\frac{3}{2})}-1=\frac{(\mu-\nu)(4k+\mu+\nu+6)}{(2k+3)(2k+2\nu+3)}.
\end{equation*}
Thus, $q_{k+1}/q_k>1$ for all $k=0,1,2,\ldots$ if $-\frac{3}{2}<\nu<\mu$.  Similarly, $q_{k+1}/q_k<1$ for all $k=0,1,2,\ldots$ if $\nu>-\frac{3}{2}$, $\mu<\nu$ and $\mu+\nu+6>0$.
 The assumptions $-\frac{3}{2}<\nu<\mu$ ($|\nu|<\mu+3$ for $-3<\mu\leq-\frac{3}{2}$, and $\mu<\nu<\mu+3$ for $\mu>-\frac{3}{2}$) ensure that $\nu>-\frac{3}{2}$, $\mu-\nu+3>0$ and $\mu+\nu+3>0$, and so all coefficients in the power series of $\mathbf{L}_\nu(x)$ and $\tilde{t}_{\mu,\nu}(x)$ are positive.  (The conditions on $\mu$ and $\nu$ in the other monotonicity results of the theorem will also always guarantee that all coefficients in the power series are positive, and we will not comment on this further in their proof.)  Finally, as the radius of convergence of the power series of $\mathbf{L}_\nu(x)$ and $\tilde{t}_{\mu,\nu}(x)$ is infinity, the assertion in part (i) follows from Lemma \ref{lem2.1}.

\vspace{2mm}

\noindent{(ii)} We proceed similarly to part (i) through an application of Lemma \ref{lem2.1}.  Consider
\begin{equation*}\frac{\mathbf{L}_\mu(x)}{\tilde{t}_{\mu,\nu}(x)}=\frac{\sum_{k=0}^\infty\alpha_{\mu,k}x^{2k}}{\sum_{k=0}^\infty\beta_{\mu,\nu,k} x^{2k}}.
\end{equation*}
Let  $q_k=\alpha_{\mu,k}/\beta_{\mu,\nu,k}$.  Then
\begin{equation*}\frac{q_{k+1}}{q_k}-1=\frac{\big(k+\frac{\mu-\nu+3}{2}\big)\big(k+\frac{\mu+\nu+3}{2}\big)}{(k+\frac{3}{2})(k+\mu+\frac{3}{2})}-1=\frac{\mu^2-\nu^2}{(2k+3)(2k+2\mu+3)}.
\end{equation*}
Thus, $q_{k+1}/q_k>1$ for all $k=0,1,2,\ldots$ if $|\mu|>|\nu|$ and $\mu>-\frac{3}{2}$, and $q_{k+1}/q_k<1$ for all $k=0,1,2,\ldots$ if $|\mu|<|\nu|$ and $\mu>-\frac{3}{2}$.  The additional conditions on $\mu$ and $\nu$ in the statement of the monotonicity result ensure that the coefficients of the powers series are positive. 

\vspace{2mm}

\noindent{(iii)} From the power series representation (\ref{Idef}) of $I_{\mu+1}(x)$ we have that
\begin{equation*}\frac{I_{\mu+1}(x)}{\tilde{t}_{\mu,\nu}(x)}=\frac{\sum_{k=0}^\infty\gamma_{\mu,k}x^{2k}}{\sum_{k=0}^\infty\beta_{\mu,\nu,k} x^{2k}},
\end{equation*}
where $\gamma_{\mu,k}=1/(k!\Gamma(k+\mu+2))$.  With $q_k=\gamma_{\mu,k}/\beta_{\mu,\nu,k}$ we have
\begin{equation*}\frac{q_{k+1}}{q_k}-1=\frac{\big(k+\frac{\mu-\nu+3}{2}\big)\big(k+\frac{\mu+\nu+3}{2}\big)}{(k+1)(k+\mu+2)}-1=\frac{(\mu+1)^2-\nu^2}{4(k+1)(k+\mu+2)},
\end{equation*}
Therefore the sequence $\{q_k\}_{k\geq0}$ is strictly increasing (decreasing) if $|\mu+1|>|\nu|$ and $\mu>-2$ ($|\mu+1|<|\nu|$ and $\mu>-2$).

\vspace{2mm}

\noindent{(iv)} This is very similar to part (i) of Theorem 2.1 of \cite{m17}, which is given for a different normalization of the modified Lommel function $t_{\mu,\nu}(x)$.  Due to the different normalization used, our result has a different range of validity.  We omit the details.

\vspace{2mm}

\noindent{(v)} We can write
\begin{equation*}\frac{x\tilde{t}_{\mu,\nu}'(x)}{\tilde{t}_{\mu,\nu}(x)}=\frac{\sum_{k=0}^\infty\delta_{\mu,\nu,k}x^{2k}}{\sum_{k=0}^\infty\beta_{\mu,\nu,k} x^{2k}},
\end{equation*}
where $\delta_{\mu,\nu,k}=(2k+\mu+1)\beta_{\mu,\nu,k}$.  As the sequence $\{\delta_{\mu,\nu,k}/\beta_{\mu,\nu,k}\}_{k\geq0}$ is strictly increasing it follows by Lemma \ref{lem2.1} that $x\mapsto x\tilde{t}_{\mu,\nu}'(x)/\tilde{t}_{\mu,\nu}(x)$ is strictly increasing on $(0,\infty)$ for  $\mu>-3$, $|\nu|<\mu+3$.

\vspace{2mm}

\noindent{(vi)} By part (iv), we have
\begin{equation*}\bigg[\frac{x^{-\nu}\tilde{t}_{\mu,\nu}(x)}{x^{-\nu-\delta}\tilde{t}_{\mu+\delta,\nu+\delta}(x)}\bigg]'=\bigg[\frac{x^{-\mu}\tilde{t}_{\mu,\nu}(x)}{x^{-\mu-\delta}\tilde{t}_{\mu+\delta,\nu+\delta}(x)}\bigg]'>0,
\end{equation*}
which by the quotient rule can be seen to be equivalent to
\[[x^{-\nu}\tilde{t}_{\mu,\nu}(x)]'[x^{-\nu-\delta}\tilde{t}_{\mu+\delta,\nu+\delta}(x)]-[x^{-\nu}\tilde{t}_{\mu,\nu}(x)][x^{-\nu-\delta}\tilde{t}_{\mu+\delta,\nu+\delta}(x)]'>0.\]
Using the differentiation formula (\ref{diffone1}) we can express this inequality in the form
\begin{align*}x^{-2\nu-\delta}[&\tilde{t}_{\mu+1,\nu+1}(x)\tilde{t}_{\mu+\delta,\nu+\delta}(x)-\tilde{t}_{\mu,\nu}(x)\tilde{t}_{\mu+\delta+1,\nu+\delta+1}(x)]\\
&>x^{\mu-\nu}\bigg[\frac{2^{-\mu-\delta}x^{-\mu}\tilde{t}_{\mu,\nu}(x)}{\Gamma(\frac{\mu-\nu+3}{2})\Gamma(\frac{(\mu+\delta)+(\nu+\delta)+3}{2})}-\frac{2^{-\mu}x^{-\mu-\delta}\tilde{t}_{\mu+\delta,\nu+\delta}(x)}{\Gamma(\frac{\mu-\nu+3}{2})\Gamma(\frac{\mu+\nu+3}{2})}\bigg].
\end{align*}
By inequality (\ref{monineq3.5}) we have that the right-hand side is positive, proving the assertion.

\vspace{2mm}

\noindent{(vii)} Inequalities (\ref{monineq1})--(\ref{monineq3.75}), and their reverses, follows from parts (i)--(v), respectively, together with a calculation of the ratios $x^{\mu-\nu}\mathbf{L}_\nu(x)/\tilde{t}_{\mu,\nu}(x)$, $\mathbf{L}_\mu(x)/\tilde{t}_{\mu,\nu}(x)$, $I_{\mu+1}(x)/\tilde{t}_{\mu,\nu}(x)$, $x^{\mu_1-\mu}\tilde{t}_{\mu,\nu}(x)/\tilde{t}_{\mu_1,\nu_1}(x)$ and $x\tilde{t}_{\mu,\nu}'(x)/\tilde{t}_{\mu,\nu}(x)$ in the limit $x\downarrow0$ using the limiting forms (\ref{ttend0}), (\ref{itend0}) and (\ref{Ltend0}).  Finally, we prove inequality (\ref{monineq4}).  Recall that $t_{\mu,\nu}(x)$ is a particular solution to (\ref{lommelode}).  On using the normalization (\ref{tseries}) that expresses $\tilde{t}_{\mu,\nu}(x)$ in terms of $t_{\mu,\nu}(x)$ and recalling the notation (\ref{anot}) for $a_{\mu,\nu}(x)$ we see that $\tilde{t}_{\mu,\nu}(x)$ satisfies
\begin{equation*}\tilde{t}_{\mu,\nu}''(x)=\bigg(1+\frac{\nu^2}{x^2}\bigg)\tilde{t}_{\mu,\nu}(x)-\frac{1}{x}\tilde{t}_{\mu,\nu}'(x)+\frac{\mu+\nu+1}{x}a_{\mu,\nu}(x),
\end{equation*}
and therefore
\begin{align*}\frac{1}{x}[\tilde{t}_{\mu,\nu}(x)]^2\bigg[\frac{x\tilde{t}_{\mu,\nu}'(x)}{\tilde{t}_{\mu,\nu}(x)}\bigg]'=\bigg(1+\frac{\nu^2}{x^2}\bigg)[\tilde{t}_{\mu,\nu}(x)]^2-[\tilde{t}_{\mu,\nu}'(x)]^2+\frac{\mu+\nu+1}{x}a_{\mu,\nu}(x)\tilde{t}_{\mu,\nu}(x).
\end{align*}
By part (v), the right-hand side of the above equation is positive, and on rearranging and recalling that $b_{\mu,\nu}(x)=\frac{xa_{\mu,\nu}(x)}{2\tilde{t}_{\mu,\nu}(x)}$ we obtain inequality (\ref{monineq4}).
\end{proof}

\begin{remark}
\begin{itemize}

\item [(i)] Parts (iii)--(vi) generalise monotonicity results for the modified Struve function $\mathbf{L}_\nu(x)$ given in Theorem 2.2 of \cite{bp14}.  Indeed, the results exactly reduce to those of \cite{bp14} in the case $\mu=\nu$. 
 
\item[(ii)] Inequality (\ref{monineq3}) (and its reverse) complement the following two-sided inequality of \cite{gaunt ineq7},
\begin{align}\label{star271}I_\nu(x)<\bigg(\frac{x^2}{(\mu+3)^2-\nu^2+x^2}\bigg)^{-\frac{\mu-\nu+1}{2}}\tilde{t}_{\mu,\nu}(x)
<C_{\mu,\nu}I_\nu(x),
\end{align}
where
\[C_{\mu,\nu}=\frac{((\mu+3)^2-\nu^2)^{\frac{\mu-\nu+1}{2}}\Gamma(\nu+1)}{2^{\mu-\nu+1}\Gamma\big(\frac{\mu-\nu+3}{2}\big)\Gamma(\frac{\mu+\nu+3}{2}\big)},\]
and both the lower and upper bounds are valid for $\mu>-2$, $-1<\nu<\mu+1$.  Using the limiting forms (\ref{ttend0}) and (\ref{itend0}) we see that inequality (\ref{monineq3}) (and its reverse) and the upper bound of (\ref{star271}) are tight in the limit $x\downarrow0$, but that the lower bound of (\ref{star271}) is not.  All bounds are of the correct asymptotic order $O(x^{-1/2}\mathrm{e}^x)$, as $x\rightarrow\infty$, but only the lower bound of (\ref{star271}) is tight in this limit (see (\ref{Itendinfinity}) and (\ref{itendinf})).  It is interesting to note that inequality (\ref{monineq3}) is expressed in terms of $I_{\mu+1}(x)$, whereas (\ref{star271}) is expressed in terms of $I_\nu(x)$.  The modified Bessel function $I_{\mu+1}(x)$ has the same asymptotic order as $\tilde{t}_{\mu,\nu}(x)$ in both the limits $x\downarrow0$ and $x\rightarrow\infty$, which helps to explain how the bound (\ref{monineq3}) has a similar performance to (\ref{star271}) despite taking a simpler form.

We also note that inequalities (\ref{monineq2}) and (\ref{monineq3}) are very useful in that they allow one to obtain a number of different bounds for $\tilde{t}_{\mu,\nu}(x)$ as a consequence of bounds in the existing literature for $\mathbf{L}_\mu(x)$ and $I_{\mu+1}(x)$.  For reasons of brevity, we only note one such example.  Applying inequality (4.59) of \cite{gaunt ineq7} to inequality (\ref{monineq2}) yields the following neat bound which further complements inequalities (\ref{monineq3}) and (\ref{star271}): for $x>0$,
\begin{equation*}\tilde{t}_{\mu,\nu}(x)<\frac{\Gamma(\mu+1)\sqrt{3(2\mu+3)}}{2\Gamma(\frac{\mu-\nu+3}{2})\Gamma(\frac{\mu+\nu+3}{2})}\frac{xI_\mu(x)}{\sqrt{x^2+3(2\mu+3)}}, \quad \mu>-1, \:|\mu|>|\nu|.
\end{equation*}
This inequality is also tight as $x\downarrow0$ and has the correct asymptotic order as $x\rightarrow\infty$.


\item [(iii)] A similar upper bound for the condition number $|x\tilde{t}_{\mu,\nu}'(x)/\tilde{t}_{\mu,\nu}(x)|$ was obtained in inequality (4.46) of \cite{gaunt ineq7}: for $x>0$,
\begin{equation*}\bigg|\frac{x\tilde{t}_{\mu,\nu}'(x)}{\tilde{t}_{\mu,\nu}(x)}\bigg|<\sqrt{\big(\nu+\tfrac{1}{2}\big)^2+x^2}+2b_{\mu,\nu}(x)-\tfrac{1}{2}, \quad \mu>-\tfrac{3}{2},\: -\tfrac{1}{2}\leq\nu<\mu+1.
\end{equation*}
An asymptotic analysis of the bounds using the limiting forms (\ref{ttend0}) and (\ref{Itendinfinity}) shows that our bound (\ref{monineq4}) outperforms the bound of \cite{gaunt ineq7} in the limit $x\downarrow0$, whilst the reverse is true as $x\rightarrow\infty$.

\end{itemize}
\end{remark}

\begin{remark}We note some consequences of inequality (\ref{monineq3.75}), which we now recall:
\begin{equation}\label{monineq3.8}\frac{u\tilde{t}_{\mu,\nu}'(u)}{\tilde{t}_{\mu,\nu}(u)}>\mu+1, \quad u>0,\: \mu>-3, \: |\nu|<\mu+3.
\end{equation}
On integrating both sides of (\ref{monineq3.8}) between $x$ and $y$ we obtain, 
\begin{equation}\label{lomineq90}\frac{\tilde{t}_{\mu,\nu}(x)}{\tilde{t}_{\mu,\nu}(y)}<\bigg(\frac{x}{y}\bigg)^{\mu+1}, \quad 0<x<y,\: \mu>-3, \: |\nu|<\mu+3.
\end{equation}
Also, from (\ref{Iidentity}) and (\ref{lomrel2}) we obtain the identity
\begin{equation*}x\tilde{t}_{\mu,\nu}'(x)+\nu\tilde{t}_{\mu,\nu}(x)=x\tilde{t}_{\mu-1,\nu-1}(x),
\end{equation*}
and combining with inequality (\ref{monineq3.8}) yields the inequality
\begin{equation}\label{lomineq91}\frac{\tilde{t}_{\mu,\nu}(x)}{\tilde{t}_{\mu-1,\nu-1}(x)}<\frac{x}{\mu+\nu+1}, \quad x>0,\: \mu>-3, \: |\nu|<\mu+3.
\end{equation}
This inequality is sharp in the limit $x\downarrow0$.  Note that inequality (\ref{lomineq91}) also follows as a special case of inequality (\ref{monineq3.5}).  Inequalities (\ref{lomineq90}) and (\ref{lomineq91}) generalise bounds of \cite{bp14} involving the modified Struve function $\mathbf{L}_\nu(x)$.  The bounds of \cite{bp14} improved the range of validity of earlier results of \cite{jn98}.

A number of inequalities for the quantities $\tilde{t}_{\mu,\nu}(x)/\tilde{t}_{\mu,\nu}(y)$ and $\tilde{t}_{\mu,\nu}(x)/\tilde{t}_{\mu-1,\nu-1}(x)$ were obtained by \cite{gaunt ineq7}.  The simple bounds obtained in this remark have the advantage over those of \cite{gaunt ineq7} by having a larger range of validity.  Also, unlike inequality (\ref{lomineq91}), none of the upper bounds of \cite{gaunt ineq7} for $\tilde{t}_{\mu,\nu}(x)/\tilde{t}_{\mu-1,\nu-1}(x)$ are sharp in the limit $x\downarrow0$.  However, the bounds of \cite{gaunt ineq7} perform much better for `large' $x$ than the bounds given in this remark. 
 
\end{remark}

We now obtain a further monotonicity result and associated  inequality that complements an inequality of \cite{gaunt ineq7}.

\begin{proposition}Suppose that 
\begin{equation}\label{cond1}\nu>\max\bigg\{-\mu-2,\;\sqrt{\frac{2(\mu+3)}{5}}-\mu-3,\;-\frac{5(\mu+3)}{7}\bigg\},
\end{equation}
or
\begin{equation}\label{cond2}\nu>\max\bigg\{-\mu-2,\;\sqrt{\frac{2(\mu+4)}{7}}-\mu-3,\;-\frac{5(\mu+3)}{7}\bigg\},
\end{equation}
with the range of validity in (\ref{cond2}) being better than in (\ref{cond1}) for $\mu>-\frac{1}{2}$. Then the function $x\mapsto x^{-\mu}\tilde{t}_{\mu,\nu}(x)/\sinh(x/(\mu+\nu+3))$ is strictly increasing  on $(0,\infty)$.  Consequently, for either of these ranges of validity we have, for $x>0$,
\begin{equation}\label{sinhineq2}\tilde{t}_{\mu,\nu}(x)> \frac{(\mu+\nu+3)x^\mu\sinh\big(\frac{x}{\mu+\nu+3}\big)}{2^{\mu+1}\Gamma\big(\frac{\mu-\nu+3}{2}\big)\Gamma\big(\frac{\mu+\nu+3}{2}\big)},
\end{equation}
\end{proposition}

\begin{proof}Consider the quotient
\begin{equation*}Q_{\mu,\nu}(x)=\frac{x^{-\mu}\tilde{t}_{\mu,\nu}(x)}{\sinh(\frac{x}{\mu+\nu+3})}=\frac{\sum_{k=0}^\infty \beta_{\mu,\nu,k}(\frac{1}{2}x)^{2k+1}}{\sum_{k=0}^\infty \epsilon_{\mu,\nu,k}(\frac{1}{2}x)^{2k+1}},
\end{equation*}
where $\beta_{\mu,\nu,k}$ is defined as in (\ref{betaterm}) and 
\[\epsilon_{\mu,\nu,k}=\frac{1}{(2k+1)!(\frac{\mu+\nu+3}{2})^{2k+1}}.\] 
Letting $q_{k}=\beta_{\mu,\nu,k}/\epsilon_{\mu,\nu,k}$ we have
\begin{equation*}\frac{q_{k+1}}{q_k}=\frac{(2k+2)(2k+3)(\frac{\mu+\nu+3}{2})^2}{(k+\frac{\mu-\nu+3}{2})(k+\frac{\mu+\nu+3}{2})},
\end{equation*}
and a short calculation shows that $q_{k+1}/q_k>1$ is equivalent to
\begin{align}\label{quadratic}P_{\mu,\nu}(k):=(\mu+\nu+2)(\mu+\nu+4)k^2+\frac{5}{2}\bigg[(\mu+\nu+3)^2-\frac{2}{5}(\mu+3)\bigg]k&\nonumber\\
+\frac{1}{4}(\mu+\nu+3)(5\mu+7\nu+15)>0&
\end{align}
for all $k=0,1,2,\ldots$.  It can be checked that the conditions on $\mu$ and $\nu$ in (\ref{cond1}) ensure that all coefficients of this quadratic are non-negative and that at least one of them is strictly positive (and that the coefficients of the power series are positive), meaning that $q_{k+1}/q_k>1$ for all $k=0,1,2,\ldots$.  Appealing to Lemma \ref{lem2.1} thus proves the monotonicity result (for the range of validity (\ref{cond1})), and inequality (\ref{sinhineq2}) is then deduced from this and a computation of the limit $\lim_{x\downarrow0}Q_{\mu,\nu}(x)$ using the limiting form (\ref{ttend0}).

We can verify that the monotonicity result (and thus inequality (\ref{sinhineq2})) is also valid in the alternative range of validity (\ref{cond2}) using a successive differences method. To ensure that $P_{\mu,\nu}(k)>0$ for all $k=0,1,2,\ldots$ it suffices to show that $P_{\mu,\nu}(0)>0$ and $\Delta P_{\mu,\nu}(k):=P_{\mu,\nu}(k+1)-P_{\mu,\nu}(k)\geq0$ for all $k=0,1,2,\ldots$. A short calculation gives that
\begin{align*}\Delta P_{\mu,\nu}(k)&=P_{\mu,\nu}(k+1)-P_{\mu,\nu}(k)\\
&=2(\mu+\nu+2)(\mu+\nu+4)k+\frac{7}{2}(\mu+\nu+3)^2-(\mu+4).
\end{align*}
We are thus guaranteed that $P_{\mu,\nu}(k)>0$ for all $k=0,1,2,\ldots$ if
\begin{align}\label{cond3}\mu+\nu+2\geq0, \quad \frac{7}{2}(\mu+\nu+3)^2-(\mu+4)\geq0, \quad 5\mu+7\nu+15>0.
\end{align}
Note that these conditions also ensure that the coefficients of the power series are positive. From condition (\ref{cond3}) we deduce condition (\ref{cond2}). Finally, the assertion that the range of validity (\ref{cond2}) is better than (\ref{cond1}) for $\mu>-\frac{1}{2}$ follows because in this case $2(\mu+3)/5>2(\mu+4)/7$.
\end{proof}


\begin{remark}

\begin{itemize}

\item [(i)] Inequality (\ref{sinhineq2}) generalises a lower bound of \cite[p$.$ 902]{bp14} for the modified Struve function $\mathbf{L}_\nu(x)$ (and in fact in this $\mu=\nu$ case the inequality improves on the bound of \cite{bp14} by having an additional factor of $(2\nu+3)$ in the numerator of the multiplicative constant). Inequality (\ref{sinhineq2}) also complements the following inequality given in Proposition 3.2 of \cite{gaunt ineq7}, which states that, for $x>0$,
\begin{equation*}\tilde{t}_{\mu,\nu}(x)\leq \frac{x^\mu\sinh(x)}{2^{\mu+1}\Gamma\big(\frac{\mu-\nu+3}{2}\big)\Gamma\big(\frac{\mu+\nu+3}{2}\big)},
\end{equation*}
which holds for $\mu\geq-\frac{1}{2}$, $(\mu+3)^2-\nu^2\geq6$, and the inequality is reversed if $-3<\mu\leq-\frac{1}{2}$, $(\mu+3)^2-\nu^2\leq6$, $|\nu|<\mu+3$, with equality if and only if $\mu=\nu=-\frac{1}{2}$.

\item[(ii)] As $P_{\mu,\nu}(k)$ is a quadratic polynomial when $(\mu+\nu+2)(\mu+\nu+3)\not=0$, if $\mu+\nu+2>0$ and $P_{\mu,\nu}(0)<0$, that is $-\mu-2<\nu<-5(\mu+3)/7$, then there exists $m\geq1$ such that $P_{\mu,\nu}(k)<0$ for $0\leq k< m$ and $P_{\mu,\nu}(k)\geq0$ for $k\geq m$. Therefore the sequence $\{q_k\}_{k\geq0}=\{\beta_{\mu,\nu,k}/\epsilon_{\mu,\nu,k}\}_{k\geq0}$ is decreasing for $0\leq k\leq m$ and increasing for $k\geq m$.  In this case, Lemma \ref{lem2.1} cannot be used to assess the monotonicity of the ratio $Q_{\mu,\nu}(x)=x^{-\mu}\tilde{t}_{\mu,\nu}(x)/\sinh(x/(\mu+\nu+3))$.  However, we note that the technique of \cite[Theorem 2.1 and Corollary 2.3]{yang} can be used to show that, for $-\mu-2<\nu<-5(\mu+3)/7$, there exists $x_0>0$ such that $Q_{\mu,\nu}(x)$ is strictly decreasing on $(0,x_0)$ and strictly increasing on $(x_0,\infty)$.

\end{itemize}

\end{remark}

\subsection{Tur\'{a}n type inequalities}\label{sec3.2}

\begin{theorem}\label{turanthm}Let $\Delta_{\mu,\nu}(x)=[\tilde{t}_{\mu,\nu}(x)]^2-\tilde{t}_{\mu-1,\nu-1}(x)\tilde{t}_{\mu+1,\nu+1}(x)$. 

\begin{itemize}
\item[(i)] Suppose that $\mu>-3$ and $|\nu|<\mu+3$.  Then, for $x>0$,
\begin{align}\label{turan3}0<\frac{2(\frac{1}{2}x)^{2\mu+2}}{(\mu+\nu+3)\big[\Gamma(\frac{\mu-\nu+3}{2})\Gamma(\frac{\mu+\nu+3}{2})\big]^2}\leq\Delta_{\mu,\nu}(x)\leq\frac{2[\tilde{t}_{\mu,\nu}(x)]^2}{\mu+\nu+3}.
\end{align}

\item[(ii)] For $x>0$,
\begin{align}\frac{1}{\Gamma(\frac{\mu-\nu+1}{2})\Gamma(\frac{\mu+\nu+3}{2})}\frac{2(\frac{1}{2}x)^{\mu+1}\tilde{t}_{\mu,\nu}(x)}{\nu+\frac{1}{2}+\sqrt{(\nu+\frac{1}{2})^2+x^2}}&<\Delta_{\mu,\nu}(x)<\nonumber\\
\label{turan4}&\!\!\!\!\!\!\!\!\!\!\!\!\!\!\!\!\!\!\!\!<\frac{(\mu-\nu+4)[\tilde{t}_{\mu,\nu}(x)]^2}{\mu+\frac{3}{2}+\sqrt{(\nu+\frac{3}{2})^2+x^2}},
\end{align}
where the lower bound is valid for $\mu>-1$, $0\leq\nu<\mu+1$, and the upper bound holds for $\mu>-1$, $\frac{1}{2}\leq\nu<\mu+1$.

\end{itemize}
\end{theorem}

We will need the following lemma \cite[Remark 3]{kk13}.

\begin{lemma}\label{loglem}Consider the function $\delta\mapsto g(\delta,x)=\sum_{k=0}^\infty\frac{b_kx^k}{\Gamma(k+\delta)}$.
Suppose that the non-trivial and non-negative sequence $\{b_k\}_{k\geq0}$ is log-concave and has no internal zeros.  Then $\delta\mapsto g(\delta,x)$ satisfies the Tur\'{a}n type inequality
\begin{equation*}\frac{b_0^2}{\delta[\Gamma(\delta)]^2}\leq[g(\delta,x)]^2-g(\delta-1,x)g(\delta+1,x)\leq\frac{1}{\delta}[g(\delta,x)]^2, \quad x>0,\:\delta>0.
\end{equation*}
\end{lemma}

\noindent{\emph{Proof of Theorem \ref{turanthm}.} (i) We apply Lemma \ref{loglem} with log-concave sequence $\{b_k\}_{k\geq0}$ defined by $b_k=1/\Gamma(k+\frac{\mu-\nu+3}{2})$, and $g(\delta+\frac{\mu'+\nu'+3}{2},\frac{1}{4}x^2)=g(\frac{(\mu'+\delta)+(\nu'+\delta)+3}{2},\frac{1}{4}x^2)=(\frac{1}{2}x)^{-\mu'-\delta-1}\tilde{t}_{\mu'+\delta,\nu'+\delta}(x)$ to obtain
\begin{align*}&\frac{2(\frac{1}{2}x)^{2\mu'+2\delta+2}}{((\mu'+\delta)+(\nu'+\delta)+3)\big[\Gamma(\frac{\mu'-\nu'+3}{2})\Gamma(\frac{(\mu'+\delta)+(\nu'+\delta)+3}{2})\big]^2}\\
&\quad\quad\quad\quad\quad\quad\quad\quad\quad\quad\quad\quad\leq\Delta_{\mu',\nu'}(x)\leq\frac{2[\tilde{t}_{\mu'+\delta,\nu'+\delta}(x)]^2}{(\mu'+\delta)+(\nu'+\delta)+3}.
\end{align*}
Now let $\mu=\mu'+\delta$ and $\nu=\nu'+\delta$.

\vspace{2mm}

\noindent{(ii)} We begin by noting the following two-sided inequality of \cite{gaunt ineq7}:
\begin{equation}\label{upper1}\frac{x}{\nu-\frac{1}{2}+2b_{\mu,\nu}(x)+\sqrt{\big(\nu+\frac{1}{2}\big)^2+x^2}}<\frac{\tilde{t}_{\mu,\nu}(x)}{\tilde{t}_{\mu-1,\nu-1}(x)}<\frac{x}{\nu-\frac{1}{2}+\sqrt{\big(\nu-\frac{1}{2}\big)^2+x^2}},
\end{equation}
where the lower bound is valid for $\mu>-1$, $0\leq\nu<\mu+1$ and the upper bound is valid for $\mu>-\frac{1}{2}$, $\frac{1}{2}\leq\nu<\mu+1$.  From (\ref{upper1}) it follows that, for $\mu>-1$, $0\leq\nu<\mu+1$,
\begin{align}\label{turan1}\frac{\tilde{t}_{\mu-1,\nu-1}(x)}{\tilde{t}_{\mu,\nu}(x)}\frac{\tilde{t}_{\mu+1,\nu+1}(x)}{\tilde{t}_{\mu,\nu}(x)}<\frac{\nu-\frac{1}{2}+2b_{\mu,\nu}(x)+\sqrt{\big(\nu+\frac{1}{2}\big)^2+x^2}}{\nu+\frac{1}{2}+\sqrt{\big(\nu+\frac{1}{2}\big)^2+x^2}},
\end{align}
and, for $\mu>-1$, $\frac{1}{2}\leq\nu<\mu+1$,
\begin{align}\label{turan2}\frac{\tilde{t}_{\mu-1,\nu-1}(x)}{\tilde{t}_{\mu,\nu}(x)}\frac{\tilde{t}_{\mu+1,\nu+1}(x)}{\tilde{t}_{\mu,\nu}(x)}>\frac{\nu-\frac{1}{2}+\sqrt{\big(\nu-\frac{1}{2}\big)^2+x^2}}{\nu+\frac{1}{2}+2b_{\mu+1,\nu+1}(x)+\sqrt{\big(\nu+\frac{3}{2}\big)^2+x^2}}.
\end{align}
Rearranging (\ref{turan1}) gives
\[\Delta_{\mu,\nu}(x)>\frac{2b_{\mu,\nu}(x)[\tilde{t}_{\mu,\nu}(x)]^2}{\nu+\frac{1}{2}+\sqrt{(\nu+\frac{1}{2})^2+x^2}},\]
as required. On the other hand, first using the inequality $2b_{\mu+1,\nu+1}(x)<\mu-\nu+1$ (see \cite{gaunt ineq7}, part (iii) of Lemma 3.3), and then rearranging (\ref{turan2}) gives
\begin{align*}\frac{\Delta_{\mu,\nu}(x)}{[\tilde{t}_{\mu,\nu}(x)]^2}<\frac{\mu-\nu+2+\sqrt{(\nu+\frac{3}{2})^2+x^2}-\sqrt{(\nu-\frac{1}{2})^2+x^2}}{\mu+\frac{3}{2}+\sqrt{(\nu+\frac{3}{2})^2+x^2}}<\frac{\mu-\nu+4}{\mu+\frac{3}{2}+\sqrt{(\nu+\frac{3}{2})^2+x^2}},
\end{align*}
as required. \hfill $\Box$

\begin{remark}
\begin{itemize}

\item [(i)] The two-sided inequality (\ref{turan3}) generalises a two-sided Tur\'{a}n type inequality of \cite[p$.$ 903]{bp14} for the modified Struve function $\mathbf{L}_\nu(x)$, but there is no modified Struve analogue of the double inequality (\ref{turan4}).  It should be noted that several other bounds for the ratio $\tilde{t}_{\mu,\nu}(x)/\tilde{t}_{\mu-1,\nu-1}(x)$ are given in \cite{gaunt ineq7}, and proceeding similarly to the proof of part (ii) of Theorem \ref{turanthm} with these bounds would lead to alternative bounds for $\Delta_{\mu,\nu}(x)$.

Using (\ref{ttend0}) one can compute $\lim_{x\downarrow0}\Delta_{\mu,\nu}(x)/\tilde{t}_{\mu,\nu}(x)=2/(\mu+\nu+3)$, and so the upper bound in (\ref{turan3}) is sharp as $x\downarrow0$.  None of the other bounds are sharp in this limit, but they are all of the correct asymptotic order $O(x^{2\mu+2})$.  Also, using (\ref{Itendinfinity}) we have that $\Delta_{\mu,\nu}(x)\sim\frac{1}{2\pi x^2}\mathrm{e}^{2x}$, as $x\rightarrow\infty$.  Since, $\tilde{t}_{\mu,\nu}(x)\sim\frac{1}{\sqrt{2\pi x}}\mathrm{e}^{x}$, as $x\rightarrow\infty$, we find that the upper bound in (\ref{turan4}) is of the correct asymptotic order, but not sharp, in this limit.  All other bounds in the theorem are, however, not of the correct asymptotic order.

\item [(ii)] The upper bound in inequality (\ref{turan3}) also complements two one-sided Tur\'{a}n type inequalities of \cite{m17}, which were given for the normalized modified Lommel function
\begin{align*}\lambda_{\mu,\nu}(x)&=(\mu-\nu+1)(\mu+\nu+1)x^{-\mu-1}t_{\mu,\nu}(x) \\
&=2^{\mu+1}\Gamma\big(\tfrac{\mu-\nu+3}{2}\big)\Gamma\big(\tfrac{\mu+\nu+3}{2}\big)x^{-\mu-1}\tilde{t}_{\mu,\nu}(x).
\end{align*}
For $x>0$,
\begin{align*}[\lambda_{\mu,\nu}(x)]^2-\lambda_{\mu-1,\nu}(x)\lambda_{\mu+1,\nu}(x)&\leq 0, \quad \mu>0, \:|\nu|<\mu+3,\\
[\lambda_{\mu,\nu}(x)]^2-\lambda_{\mu,\nu-1}(x)\lambda_{\mu,\nu-1}(x)&\leq 0, \quad \mu>-1, \:|\nu|<\mu+3,
\end{align*}
which when expressed in terms of $\tilde{t}_{\mu,\nu}(x)$ read
\begin{align}\label{turan5}[\tilde{t}_{\mu,\nu}(x)]^2-\tilde{t}_{\mu-1,\nu}(x)\tilde{t}_{\mu+1,\nu}(x)&\leq A_{\mu,\nu}[\tilde{t}_{\mu,\nu}(x)]^2,\\
\label{turan6}[\tilde{t}_{\mu,\nu}(x)]^2-\tilde{t}_{\mu,\nu-1}(x)\tilde{t}_{\mu,\nu+1}(x)&\leq A_{\mu,\nu}[\tilde{t}_{\mu,\nu}(x)]^2,
\end{align}
where
\[A_{\mu,\nu}=1-\frac{\big[\Gamma(\frac{\mu-\nu+3}{2})\Gamma(\frac{\mu+\nu+3}{2})\big]^2}{\Gamma(\frac{\mu-\nu+2}{2})\Gamma(\frac{\mu+\nu+2}{2})\Gamma(\frac{\mu-\nu+4}{2})\Gamma(\frac{\mu+\nu+4}{2})}.\]
A similar asymptotic analysis to the one carried out in part (i) show that (\ref{turan5}) and (\ref{turan6}) are tight in the limit $x\downarrow0$, as is the case for the upper bound of inequality (\ref{turan3}).  It remains an open problem, however, to obtain lower bounds for $[\tilde{t}_{\mu,\nu}(x)]^2-\tilde{t}_{\mu-1,\nu}(x)\tilde{t}_{\mu+1,\nu}(x)$ and $[\tilde{t}_{\mu,\nu}(x)]^2-\tilde{t}_{\mu,\nu-1}(x)\tilde{t}_{\mu,\nu+1}(x)$.

\end{itemize}
\end{remark}


\subsection*{Acknowledgements}
The author is supported by a Dame Kathleen Ollerenshaw Research Fellowship.

\end{document}